\documentclass[12pt]{article}
\usepackage{array, amsthm,amssymb,amsmath,graphicx,cite}
\usepackage{algorithmic, algorithm}
\usepackage[margin = 1.2in]{geometry}
\newcommand{\R}{{\mathbb R}}
\newcommand{\conv}{{\mathrm{conv}}}

\DeclareMathOperator*{\argmin}{argmin}
\DeclareMathOperator*{\argmax}{argmax}

\newtheorem{theorem}{Theorem}
\newtheorem{proposition}{Proposition}

\newtheorem{corollary}{Corollary}

\title{Enhanced Basic Procedures for the Projection and Rescaling Algorithm} 
\date{}
\author{David Huckleberry Gutman}
\begin{document}

\maketitle

%
%
\begin{abstract}
Using an efficient algorithmic implementation of Caratheodory's theorem, we propose three enhanced versions of the Projection and Rescaling algorithm's basic procedures each of which improves upon the order of complexity of its analogue in [Mathematical Programming Series A, 166 (2017), pp. 87-111].
\end{abstract}
\section{Introduction}

Pe{\~n}a and Soheili \cite{PenaS16} propose a two-step \textit{projection and rescaling algorithm}, which extends an algorithm by Chubanov \cite{Chubanov15}, to solve the conic feasibility problem
\begin{equation}\label{eq.feas}
\text{Find } x\in L\cap\R^n_{++}
\end{equation}
where $L$ subspace of $\R^n$ \cite{Chubanov15, PenaS16}. Assuming the projection matrix, $P_L$, for $L$ is available we can rewrite \eqref{eq.feas} as
\begin{equation}\label{eq.feasPro}
\text{Find } x \text{ such that }P_L x>0.
\end{equation}
The projection and rescaling algorithm consists of two subprocedures:
\begin{enumerate}
	\item \textit{Basic Procedure (Projection)}: This procedure uses $P_L$ to find a point in $L\cap \R^n_+$ provided this cone contains a deeply interior point. This is implemented via one of three schemes based on the Von Neumann/Perceptron algorithm.
	\item \textit{Rescaling}: Using the final iterate from the basic procedure, this step rescales $L\cap\R^n_+$ such that its interior points - provided $L\cap\R^n_+\neq\emptyset$ - become ``deeper'' in the interior of $\R^n_+$. 
\end{enumerate}
In this paper, we propose enhancements of each of the Von Neumann/Perceptron basic procedures. Our enhancements iteratively apply a technique used to prove Caratheodory's theorem. These enhanced procedures improve the complexity of the basic procedure from $O(n^4 m)$ to $O(n^2 m^3)$ operations when $L$ has dimension $m$: a significant improvement when $m<<n$.  

Fundamentally, the basic procedure of Pe{\~n}a and Soheili adapts the Von Neumann and Perceptron procedures to iteratively reduce $\|P_L x\|_2$ on the standard $n$ dimensional simplex $\Delta_{n-1}:=\{x\in\R^n:\|x\|_1=1,x\geq 0\}$ until either $P_L x>0$ or $\|P_L x\|_2\leq\frac{1}{3\sqrt{n}}\|x\|_\infty$.  Thus, the basic procedure intends to approximately solve the subproblem
\begin{equation}\label{eq.norm}
\min_{x\in\Delta_{n-1}}\|P_L x\|^2_2.
\end{equation}
The convergence proofs in \cite{PenaS16} depends on the observation that $\|x\|_\infty\geq\frac{1}{n}$ for all $x\in\Delta_{n-1}$. As such, the reasoning in \cite{PenaS16} yields a faster rate provided we restrict the iterates of the Von Neumann/Perceptron schemes to proper faces of $\Delta_{n-1}$. If $Q\in\R^{n\times m}$ is an orthonormal basis for $L$ then $P_L=QQ^T$ and we may rephrase \eqref{eq.norm} as
\begin{equation}\label{eq.normQ}
\min_{x\in\Delta_{n-1}}\|Q^Tx\|^2_2=\min_{z\in\conv(Q^T)}\|z\|^2_2
\end{equation}
where $\conv(Q^T)$ denotes the convex hull of $Q^T$'s columns. By Caratheodory's theorem, any point in $\conv(Q^T)$ can be written as convex combination of at most $m+1$ columns of $Q^T$.  Our proposed enhancements apply this observation to ensure that each of the iterates is a convex combinations of no more than $m+1$ columns of $Q^T$. 

Our enhanced basic procedures iteratively reduce the objective \eqref{eq.normQ} using a Von Neumann/Perceptron scheme which applies a modified version of the Incremental Representation Reduction (IRR) procedure of \cite{BeckS17} at each iteration. When provided a point $z\in\conv(Q^T)$,  the IRR outputs a new affinely independent, convex representation of the point $z$ provided it already contains a sufficiently large set of affinely independent vectors in its support. Whereas the IRR operates in $O(m^3)$ time, our version operates in $O(m^2)$ time by allowing for vectors in the representation of $x$ that have zero support.

This paper is organized as follows. Section 2 describes the mIRR, and proves important properties of it including its $O(m^2)$ complexity. Section 3 describes the limited support Von Neumann and Perceptron algorithms. Section 4 elaborates possible extensions.

%
%
\section{Modified Incremental Representation Reduction}

The heart of our improved basic procedure is a modified version of the Incremental Representation Reduction Procedure of \cite{BeckS17}. This subprocedure iteratively applies the main technique used in standard proofs of Caratheodory's theorem \cite{HiriartU2012}. To simplify notation, given a matrix $A\in\R^{m\times n}$ we define $\tilde{A}$ as the augmented matrix $\begin{bmatrix}1\dots 1\\ A \end{bmatrix}$. If $B=[B(1),...,B(k)]\subseteq \{1,...,n\}$ is an ordered set of indices and $x\in\R^n$ we let $A_B=[A_{B(1)}...A_{B(k)}]\in\R^{m\times k}$ where $A_{B(i)}$ denotes the $B(i)$-th column of $A$ and $x_B=(x_{B(i)}...x_{B(k)})\in\R^k$ and $x_{B(i)}$ denotes the $B(i)$-th entry of $x$. Given a full column rank matrix $M\in\R^{k\times l}$, we let $M^\dagger$ denote its unique pseudoinverse, $(M^TM)^{-1}M^T$. This notation strongly mimics the notation used in  \cite{Bertsimas97} to present the revised Simplex method. The resemblance is entirely intentional; our method strongly resembles the revised Simplex method.

\begin{theorem}\label{thm.inv}
Suppose $B\subseteq\{1,...,n\}$ is an ordered set of indices such that $A_{B}$ consists of affinely independent columns and $\tilde{A}_{B}^\dagger$ is known. If $z=Ax=A_B x_B+A_j x_j$ for some $x\in\Delta_{n-1}$ and $j\in\{1,...,n\}\backslash B$ then we can find $x^+\in\Delta_{n-1}$, an ordered set of indices $B^+\subseteq B':=\begin{bmatrix} B & j\end{bmatrix}$, and $\tilde{A}_{B^+}^\dagger$ such that $z=Ax=A_{B^+}x_{B^+}^+$ and $A_{B^+}$ consists of affinely independent columns in $O(m^2)$ operations.
\end{theorem}

\begin{proof}
There are two cases we must tackle:
\begin{enumerate}
	\item $\tilde{A}_j\neq\tilde{A}_B\tilde{A}_B^\dagger\tilde{A}_j$: $A_j$ is affinely {\em independent} of the columns of $A_B$, i.e. the matrix $\begin{bmatrix} A_B & A_j\end{bmatrix}$ has affinely independent columns.
	\item $\tilde{A}_j=\tilde{A}_B\tilde{A}_B^\dagger\tilde{A}_j$: $A_j$ is affinely {\em dependent} on the columns of $A_B$, i.e. the matrix $\begin{bmatrix} A_B & A_j\end{bmatrix}$ has affinely independent columns.
\end{enumerate}
Determining the equality of $\tilde{A}_j$ and $\tilde{A}_B\tilde{A}_B^\dagger\tilde{A}_j$ requires vector-matrix multiplication, an $\mathcal{O}(m^2)$ operation.
\item {\em Case 1} ($\tilde{A}_j\neq\tilde{A}_B\tilde{A}_B^\dagger\tilde{A}_j$): In this case, let $B^+=B'$ and $x^+=x$. We claim that $\tilde{A}_{B^+}$ is given by the formula
\begin{equation}\label{eq.inv}
\tilde{A}_{B^+}^\dagger=\begin{bmatrix}\tilde{A}_B^\dagger\\0 \end{bmatrix}-\begin{bmatrix}\tilde{A}_B^\dagger\tilde{A}_j\\-1 \end{bmatrix}\frac{\tilde{A}_j^{T}(I-\tilde{A}_B\tilde{A}_B^\dagger)}{\tilde{A}_j^T(I-\tilde{A}_B\tilde{A}_B^\dagger)\tilde{A}_j}.
\end{equation}
The quantity $\tilde{A}_j^T(I-\tilde{A}_B\tilde{A}_B^\dagger)\tilde{A}_j=\|\tilde{A}_j-\tilde{A}_B\tilde{A}_B^\dagger\tilde{A}_j\|^2$ is non-zero by hypothesis and thus the expression on the right is well defined. It suffices to verify that right multiplication of the right hand side of \eqref{eq.inv} by $\tilde{A}_{B^+}=\begin{bmatrix}\tilde{A}_B & \tilde{A}_j\end{bmatrix}$ yields the identity matrix. We compute
\begin{align}
\begin{bmatrix}\tilde{A}_B^\dagger\\0 \end{bmatrix}\begin{bmatrix}\tilde{A}_B & \tilde{A}_j\end{bmatrix}&=\begin{bmatrix}\tilde{A}_B^\dagger \tilde{A}_B & \tilde{A}_B^\dagger \tilde{A}_j\\0 & 0\end{bmatrix}=\begin{bmatrix}I_{k\times k} & \tilde{A}_B^\dagger \tilde{A}_j\\0 & 0\end{bmatrix}\label{eq.pseudo.1}\\
\tilde{A}_j^{T}(I-\tilde{A}_B\tilde{A}_B^\dagger)\begin{bmatrix}\tilde{A}_B & \tilde{A}_j\end{bmatrix}&=\tilde{A}_j^{T}(I-\tilde{A}_B\tilde{A}_B^\dagger)\tilde{A}_j\begin{bmatrix}0 &1\end{bmatrix}\label{eq.pseudo.2}.
\end{align}
Observe that the right hand side of \eqref{eq.pseudo.2} is non-zero since $\tilde{A}_j\neq\tilde{A}_B\tilde{A}_B^\dagger\tilde{A}_j$. Combining equations \eqref{eq.pseudo.1} and \eqref{eq.pseudo.2} yields
\begin{align*}
\left(\begin{bmatrix}\tilde{A}_B^\dagger\\0 \end{bmatrix}-\begin{bmatrix}\tilde{A}_B^\dagger\tilde{A}_j\\-1 \end{bmatrix}\frac{\tilde{A}_j^{T}(I-\tilde{A}_B\tilde{A}_B^\dagger)}{\tilde{A}_j^T(I-\tilde{A}_B\tilde{A}_B^\dagger)\tilde{A}_j}\right)\begin{bmatrix}\tilde{A}_B & \tilde{A}_j\end{bmatrix}&=\begin{bmatrix}I_{k\times k} & \tilde{A}_B^\dagger \tilde{A}_j\\0 & 0\end{bmatrix}-\begin{bmatrix}0 & \tilde{A}_B^\dagger\tilde{A}_j\\  0 & -1\end{bmatrix}\\
&=I_{(k+1)\times(k+1)}
\end{align*}
thus verifying our formula for $\tilde{A}_{B^+}^\dagger$. The formula \eqref{eq.inv} uses matrix addition and vector-matrix multiplication so it takes at most $\mathcal{O}(m^2)$ operations.

\item {\em Case 2} ($\tilde{A}_j=\tilde{A}_B\tilde{A}_B^\dagger\tilde{A}_j$): Let $u=\begin{bmatrix}\tilde{A}_B^\dagger \tilde{A}_j & -1\end{bmatrix}$, $\theta^*=\max_{i: u_i<0}\frac{x_{B'(i)}}{u_i}$, $x^+_{B'}=x_{B'}+\theta^* u$, and $x^+_{(B')^c}=0$. We must show that $x^+\in\Delta_{n-1}$. By hypothesis, $u$ is the solution to the system
\[
\begin{bmatrix}\tilde{A}_B & \tilde{A}_j\end{bmatrix}u=0
\]
because $\tilde{A}_j=\tilde{A}_B\tilde{A}_B^\dagger\tilde{A}_j$. Hence, $\sum_{i=1} u_i=0$ which implies
\[
\sum_{i=1}^{k+1}x^+_{B'(i)}=\sum_{i=1}^{k+1}x_{B'(i)}=1.
\]
Moreover, the definition of $\theta^*$ ensures $x^+\geq 0$ completing our proof that $x^+\in\Delta_{n-1}$. 

Next, we construct $B^+$ and $\tilde{A}_{B^+}^\dagger$. Let $i^*$ the smallest index such that $\frac{x_{B'(i^*)}}{u_i}=\theta^*$ and $u_i<0$. By construction, $\theta^*$ ensures $x^+_{B^+(i^*)}=0$. We now have two subcases: $B(i^*)=j$ and $B(i^*)\neq j$. In the first case, let $B^+=B$ and $\tilde{A}_{B^+}^\dagger=\tilde{A}_B^\dagger$.  By hypothesis, $A_{B^+}=A_B$ consists of affinely independent columns. In the second case, let $B^+(i)=B(i)$ for $i\neq i^*$ and $B^+(i^*)=j$. We must show that $A_{B+}$ consists of affinely independent columns. Assume for the sake of contradiction that it does not. Then there must exist some $w\in\R^{k+1}$ such that $w_{i^*}=0$ and 
\[
0=\tilde{A}_{B+}w-\tilde{A}_j=\tilde{A}_B w-\tilde{A}_j
\]
but this implies that
\[
\tilde{A}_{B}\left(w-\tilde{A}_B^\dagger\tilde{A}_j\right)=\tilde{A}_j-\tilde{A}_j=0.
\]
By affine independence of the columns of $A_B$ we determine that
\[
w-\tilde{A}_B^\dagger\tilde{A}_j=0\Leftrightarrow w=\tilde{A}_B^\dagger\tilde{A}_j
\]
so that the $i^*$-th entry of $\tilde{A}_B^\dagger\tilde{A}_j$, which is precisely $u_{i^*}$, is zero: a contradiction. Thus, the columns of $A_{B^+}$ are affinely independent. Finally, we prove that it is possible to derive $A_{B^+}$ in $\mathcal{O}(m^2)$ operations in this second case. Form the augmented matrix $\left[\begin{array}{c|c} \tilde{A}_B^\dagger & \tilde{A}_B^\dagger\tilde{A}_j\end{array}\right]$. Add to each row a multiple of the $i^*$-th row to make the last column equal to the unit vector $e_{i^*}$. The first $|B|$ columns of the resultant matrix are $\tilde{A}_{B^+}^\dagger$. This requires no more than $\mathcal{O}(m^2)$ operations since at most $m$ row operations are required.
\end{proof}

The proof of this theorem immediately yields our core algorithm.

\begin{algorithm}[H]
\caption{Modified Incremental Representation Reduction Procedure (mIRR)}\label{alg.MIRR}
\begin{algorithmic}[1]
	\STATE Input: An ordered set of indices $B=\left[B(1),...,B(k)\right]\subseteq\{1,...,n\}$ such that $A_B$ is a matrix with affinely independent columns, $\tilde{A}_B^\dagger$, $j\notin B$, and $y=Ax=A_Bx_B+A_jx_j$ for some $x\in\Delta_{n-1}$ with $x_{[B,j]^c}=0$.
	\STATE Compute $u'=\tilde{A}_B^\dagger \tilde{A}_j$ and $u=\begin{bmatrix}(u')^T & -1\end{bmatrix}$. If $\tilde{A}_B u'\neq\tilde{A}_j$ then the columns of $\begin{bmatrix} \tilde{A}_B & \tilde{A}_j\end{bmatrix}$ are affinely independent. In this case, output $x^+=x$, $B^+=B\cup\{j\}$, and
\[
\tilde{A}_{B^+}^\dagger=\begin{bmatrix}\tilde{A}_B^\dagger\\0 \end{bmatrix}-\begin{bmatrix}\tilde{A}_B^\dagger\tilde{A}_j\\1 \end{bmatrix}\frac{\tilde{A}_j^T(I-\tilde{A}_B\tilde{A}_B^\dagger)}{\tilde{A}_j^T(I-\tilde{A}_B\tilde{A}_B^\dagger)\tilde{A}_j}.
\]
to complete the procedure. Otherwise, proceed to the next step.
	\STATE Let $\theta^*=\max_{i:u_i<0}\frac{x_i}{u_i}$, $i^*$ be the smallest index that for which $\theta^*$ is achieved, and
	\[
	x^+_{[B,j]}=x_{[B,j]}+\theta^* u, \quad x^+_{[B,j]^c}=0.
	\] If $i^*=j$ then output $x^+$, $B^+=B$, and $\tilde{A}_{B^+}=\tilde{A}_B$ to complete the procedure. Otherwise, proceed to the next step.
	\STATE Let $B^+=[B(1),...,B(i^*-1),j,B(i^*),...,B(k)]$. Form the $|B|\times (|B|+1)$ matrix $\begin{bmatrix}\tilde{A}_B^\dagger & u'\end{bmatrix}$. Add to each row a multiple of the $i^*$-th row to make the last column equal to the unit vector $e_{i^*}$. The first $|B|$ columns of the resultant matrix are $\tilde{A}_{B^+}^\dagger$. Output $B^+$, $x^+$ and $\tilde{A}_{B^+}^\dagger$.
\end{algorithmic}
\end{algorithm}

The following is an easy corollary of theorem \ref{thm.inv}.

\begin{corollary}\label{thm.mIRR}
The mIRR produces an affinely independent representation, $y=Ax=A_{B^+}x_{B^+}$ with $x\in\Delta_{n-1}$ and $x_{(B^+)^c}=0$,  of the input point $y=Ax$ and the pseudoinverse $\tilde{A}_{B^+}^\dagger$ in $\mathcal{O}(m^2)$ operations.
\end{corollary}

%
%
\section{Limited Support Basic Procedures}

In this section we propose each of our modified basic procedures. Recall that we assume the availability of an orthonormal basis for $L$ and that $Q\in\R^{n\times m}$ is the matrix whose columns are these basis vectors. The convergence results of \cite{PenaS16} for the original basic procedures depend upon the maximum size of the support of the iterates. If it were possible to ensure that the iterates maintained affinely independent support then the support would have maximum size $m+1$. This is the crux of our enhanced procedures and the mIRR enables us to do this. Our enhanced procedures start with a single column of the matrix $Q^T$. Then, until the stopping condition is reached, they take a Von-Neumann/Perceptron-like step - which may increase the size of the support by at most one - followed by an application of mIRR to ensure the support remains affinely independent. We will let $\{q_i\}_1^n$ denote the columns of $Q^T$ and we use $P$ in place of $P_L$.

The first two schemes, the Limited Support Von Neumann and Limited Support Perceptron, are subsumed in the following framework which we call the {\em Limited Support Scheme} (LSS). Each of these procedures is an enhancement of those found in \cite{PenaS16} using mIRR.

\begin{algorithm}[H]
\caption{Limited Support Scheme}
\begin{algorithmic}
	\STATE $x_0=e_1$, $z_0=Q^Tx_0=q_1$, $B_0=\{1\}$, $\tilde{Q}_{B_0}^\dagger=\frac{1}{\|\tilde{q}_1\|}\tilde{q}_1^T$, $t=0$
	\WHILE{$Px_t>0$ and $\|(P x_t)^+\|\geq\frac{1}{3\sqrt{n}}\|x_t\|_\infty$}
		\STATE Let $j=\argmin_{i\in[n]}\left<q_i,z_t\right>$
		\STATE $x_{t+1}'=x_t+\theta_t(e_j-x_t)$
		\STATE $z_{t+1}=Q^Tx_{t+1}'$
		\IF{$j\notin B_t$}
			\STATE $(x_{t+1}, B_{t+1}, \tilde{Q}_{B_{t+1}}^\dagger)=mIRR(x_{t+1}',B_t,j,\tilde{Q}_{B_t}^\dagger,Q_j)$
		\ELSE
			\STATE $x_{t+1}=x_{t+1}'$, $B_{t+1}=B_t$, $\tilde{Q}_{B_{t+1}}^\dagger=\tilde{Q}_{B_t}^\dagger$
		\ENDIF
		\STATE $t=t+1$
	\ENDWHILE
\end{algorithmic}
\end{algorithm}

If $\theta_t=\frac{1}{t+1}$ then the resulting procedure is referred to as the {\em Limited Support Perceptron Scheme} (LSP). If $\theta_t$ is determine by an exact line search then the resulting procedure is referred to as the {\em Limited Support Von Neumann Scheme} (LSVN).

\begin{proposition}\label{prop.LSS}
The following hold for algorithms LSP and LSVN:
\begin{enumerate}
	\item For all $t\geq 0$ such that LSP or LSVN have not halted, $\|z_t\|^2\leq\frac{1}{t}$.
	\item The stopping condition, $Px_t>0$ or $\|(Px)^+\|\leq\frac{\|x_t\|_\infty}{3\sqrt{n}}$, occurs in at most $9(m+1)^2n$ iterations.
	\item LSP and LSVN require $O(m^4 n)$ arithmetic operations.  
\end{enumerate}
\end{proposition}

\begin{proof}
Part 1 of our proposition is known from \cite{PenaS16}.

\item 2. $|\{i\in\{1,...,n\}:x_i>0\}|\leq|B_t|\leq m+1$ throughout the algorithm since the columns representing $x_t$ are affinely independent. Thus, $\|x_t\|_\infty\geq\frac{1}{m+1}$ since $x_t\in\Delta_{m}$. This implies that $\frac{1}{3\sqrt{n}}\|x\|_\infty\geq\frac{1}{3(m+1)\sqrt{n}}$. As $\|(Px)^+\|\leq\|Px\|$, we conclude from 1 that the one of the two stopping conditions occurs by $t=9(m+1)^2n$.\\

\item 3. By part 1, LSP and LSVN terminate in at most $t=9(m+1)^2n$ main iterations. Each of the operations besides mIRR requires at most $nm$ operations while theorem \ref{thm.mIRR} states mIRR requires $O(m^2)$ operations. Since $m^2\leq nm$, we conclude each iteration has computational cost $O(nm)$. Thus, the number of required arithmetic operations for LSP and LSVN is $O(m^3n^2)$. \\
\end{proof}

%
%
\subsection{Limited Support Von Neumann with Away Steps Scheme}
Here we propose a limited support variation of the Von Neumann with Away Steps scheme proposed in \cite{PenaS16}. This procedure is essentially the same as above except that it allows for ``away'' directions.

\begin{algorithm}
\caption{Limited Support Von Neumann with Away Steps Scheme (LSVN)}\label{alg.LSVNA}
\begin{algorithmic}
	\STATE $x_0=e_1$, $z_0=Q^Tx_0=q_1$, $B_0=\{1\}$, $\tilde{Q}_{B_0}^\dagger=\frac{1}{\|\tilde{q}_1\|}\tilde{q}_1^T$, $t=0$
	\WHILE{$Px_t>0$ and $\|(P x_t)^+\|\geq\frac{1}{3\sqrt{n}}\|x_t\|_\infty$}
		\STATE Let $j=\argmin_{i\in[n]}\left<q_i,z_t\right>$,$k=\argmax_{i\in[n]}\left<q_i,z_t\right>$
		\IF	{$\|z_t\|^2-\left<q_j,z_t\right>>\left<q_k,z_t\right>-\|z_t\|^2$} 
			\STATE (Regular Step) $a:=e_j-x_t; \theta_{max}=1$
		\ELSE
			\STATE (Away Step) $a:=x_t-e_k; \theta_{max}=\frac{(x_t)_j}{1-(x_t)_j}$
		\ENDIF
		\STATE $\theta_t=\argmin_{\theta\in[0,\theta_{max}]}\|P(x_t+\theta a)\|^2=\min\left\{\theta_{max},-\frac{\left<x_t,Pa\right>}{\|Pa\|^2}\right\}$
		\STATE $x_{t+1}'=x_t+\theta_t a$
		\STATE $z_{t+1}=Q^Tx_{t+1}'$
		\IF{$j\notin B_t$ and a regular step is taken}
			\STATE $(z_{t+1},x_{t+1}, B_{t+1}, \tilde{Q}_{B_{t+1}}^\dagger)=mIRR(z_{t+1},x_{t+1}',B_t,j,\tilde{Q}_{B_t}^\dagger,Q_j)$
		\ELSE
			\STATE $x_{t+1}=x_{t+1}'$, $B_{t+1}=B_t$, $\tilde{Q}_{B_{t+1}}^\dagger=\tilde{Q}_{B_t}^\dagger$
		\ENDIF
		\STATE $t=t+1$
	\ENDWHILE
\end{algorithmic}
\end{algorithm}

\begin{proposition}\label{prop.LSVNA.conv}
The following hold for algorithm LSVNA:
\begin{enumerate}
	\item For all $t\geq 0$ such that LSVNA has not halted, $\|z_t\|^2\leq\frac{1}{t}$.
	\item The stopping condition, $Px_t>0$ or $\|(Px)^+\|\leq\frac{\|x_t\|_\infty}{3\sqrt{n}}$, occurs in at most $9(m+1)^2n$ iterations.
	\item LSVNA requires $O(m^4 n)$ arithmetic operations.  
\end{enumerate}
\end{proposition}

\begin{proof}
The proof of the first part of the proposition is known from \cite{PenaS16}. The remaining parts follows from similar reasoning to that in the proof of proposition \ref{prop.LSS}.
\end{proof}

%
%
\section{Extensions}

The original paper \cite{PenaS16} extends the problem \eqref{eq.feas} to the semidefinite cone and the more general symmetric cone case. It is maybe possible to extend our enhanced algorithms to these settings. The primary obstacle will be defining the systems of equations necessary to mIRR.
\bibliographystyle{plain}
\bibliography{Sparse_Basic_Procedure_ARXIV_180716}
\end{document}